\documentclass[preprint,12pt,times]{article}
\usepackage{graphicx}
\usepackage{amsthm}
\usepackage{amssymb}
\usepackage{amsmath}
\usepackage{amscd}
\usepackage{bbm}
\usepackage{float}
\usepackage{booktabs}
\numberwithin{equation}{section}
\newtheorem{definition}{Definition}[section]
\newtheorem{theorem}{Theorem}[section]

\usepackage[numbers,sort&compress]{natbib}

\numberwithin{figure}{section}
\numberwithin{table}{section}

\allowdisplaybreaks[2]

\newcommand\btd{\raise 2pt \hbox{$\hat\bigtriangledown$}\hskip 1.5pt}
\newcommand\bt{\raise 2pt \hbox{$\bigtriangledown$}\hskip 1.5pt}

\begin{document}
\date{}
\title{Existence and stability of  solitary waves to  the rotation-Camassa-Holm  equation }
\author{Hao Tong,~~~Shaojie Yang\thanks{Corresponding author:
 th06101112@163.com (Hao Tong);   ~~shaojieyang@kust.edu.cn (Shaojie Yang)  } \\~\\
\small ~ Department of  Mathematics, ~Kunming University of Science and Technology,  \\
\small ~Kunming, Yunnan 650500, China}

\date{}
\maketitle
%\noindent\rule{15.5cm}{0.5pt}
\begin{abstract}

In this paper, we  investigate existence and stability of  solitary waves to  the rotation-Camassa-Holm  equation
which can be considered as a model in the shallow water for the long-crested waves propagating near the equator with effect of the Coriolis force due to the Earth’s rotation.  We prove existence of solitary  waves  by performing a phase plane analysis.  Moreover, utilizing the approach proposed by Grillakis-Shatah-Strauss,   we prove stability of solitary waves.  \\


 \end{abstract}
\noindent\rule{15.5cm}{0.5pt}

\section{Introduction}\label{sec1}
The rotation-Camassa-Holm (R-CH) equation
\begin{equation}\label{eq.10}
\begin{split}
\begin{cases}
m_t+um_x+2u_xm+cu_x-\frac{\beta_0}{\beta}u_{xxx}+\frac{w_1}{\alpha^2}u^2u_x+\frac{w_2}{\alpha^3}u^3u_x=0,~x\in \mathbb{R},~t>0,\\
m=u-u_{xx},~x\in \mathbb{R},~t>0,\\
\end{cases}
\end{split}
\end{equation}
was implied in \cite{r20,r1,r2} as a model equation which describes the motion of the fluid with the Coriolis effect from the incompressible shallow water in the equatorial region. It is known that the ocean dynamics near the Equator is quite different from that in non-equatorial regions since the meridional component of the Coriolis force (an effect of the Earth's rotation) vanishes at the Equator, so that the Equator acts as a wave guide, facilitating azimuthal wave propagation \cite{r6,r7}.
The constants in the R-CH equation \eqref{eq.10} are defined by
\begin{align*}
&c=\sqrt{1+\Omega^2}-\Omega, \alpha=\frac{c^2}{1+c^2}, \beta_0=\frac{c(c^4+6c^2-1)}{6(1+c^2)}, \beta=\frac{3c^4+8c^2-1}{6(1+c^2)},\\
&w_1=\frac{-3c(c^2-1)(c^2-2)}{2(1+c^2)^3}, w_2=\frac{(c^2-2)(c^2-1)^2(8c^2-1)}{2(1+c^2)^5}
\end{align*}
with the parameter $\Omega$ as the constant Coriolis frequency caused by the Earth's rotation. The R-CH equation \eqref{eq.10} has the nonlocal cubic and even quartic nonlinearities and a formal Hamiltonian structure, and  has corresponding  conserved quantities as follows
\begin{align*}
I(u)=&\int_{\mathbb{R}}u dx,\\
E(u)=&\frac{1}{2} \int_{\mathbb{R}} (u^2+u_x^2)dx,\\
F(u)=&\frac{1}{2} \int_{\mathbb{R}} \left(cu^2+u^3+\frac{\beta_0}{\beta}u_x^2+\frac{w_1}{6\alpha^2}u^4+\frac{w_2}{10\alpha^3}u^5+uu_x^2\right)dx.
\end{align*}
Recently, global existence and uniqueness of the energy conservative weak solutions in the energy space $H^1$ to the R-CH equation have been derived in Ref.\cite{r3}.  Well-posedness, travelling waves and geometrical aspects  have been studied in Ref.\cite{r34}.
Non-uniform dependence and well-posedness in the sense of Hadamard have been proved in Ref.\cite{r36}.
Generic regularity of conservative solutions have been studied in Ref.\cite{r17}.  Wang-Yang-Han \cite{R35} proved that symmetric waves to the R-CH equation must be traveling waves.

In the case that the Coriolis effect vanishes (i.e., $\Omega=0$), then it's easy to observe that the coefficients in the higher-power nonlinearities $w_1=w_2=0$.  When $\Omega=0$,  the R-CH Eq.\eqref{eq.10} is reduced the remarkable Camassa-Holm  (CH) equation
\begin{align}\label{a6}
m_t+um_x+2u_xm=0,~~m=u-u_{xx},
\end{align}
which was first derived formally by Fokas and Fuchssteiner in \cite{R10},  and later derived as a model for unidirectional propagation of shallow water over a flat bottom by Camassa and Holm in \cite{R9}.
The CH equation has been studied extensively in the last twenty years because of its many remarkable properties:
infinity of conservation laws and completely integrability \cite{R10,R9}, peakons \cite{R9}, well-posedness \cite{r26,r27,r28}, wave breaking \cite{r29,R30,r31,r32}, orbital stability \cite{r21,r23},  global conservative solutions and dissipative solutions \cite{r24,r25}.

In the present paper, we are concerned with the existence and stability of  solitary waves to  the R-CH equation \eqref{eq.10}.   The study of the existence of solitary waves  to the R-CH equation \eqref{eq.10} by using dynamical systems method which have been used in studying solitary traveling water waves of moderate amplitude \cite{R24}. Moreover, we are interested in the spectral properties of solitary waves and that their shapes are stable under small disturbances, cf. Definition \ref{D1} below.  The  approach of studying stability  is inspired by  Constantin-Strauss \cite{r21}, which shows that solitary waves of the CH equation \eqref{a6} are solitons and that they are orbitally stable. In a similar way, we take advantage of an approach proposed by Grillakis-Shatah-Strauss \cite{r30} and prove the convexity of a scalar function, which is based on two conserved quantities, to deduce orbital stability of solitary waves.

The rest of this paper is organized as follows. In Section \ref{sec2}, we discuss the existence of solitary waves to the R-CH equation. In Section \ref{sec3}, we prove that solitary waves of the R-CH equation are orbitally stable.

\section{Existence of solitary  waves}\label{sec2}
In this section, we  discuss the existence of solitary waves to the R-CH equation
by performing a phase plane analysis.

 For a traveling wave solution $u(t,x)=\varphi(x-\sigma t)$ with speed $\sigma$  satisfies the following equation\\
\begin{equation}\label{eq.y}
-\sigma\varphi_x+\sigma\varphi_{xxx}+c\varphi_x+3\varphi\varphi_x-\frac{\beta_0}{\beta}\varphi_{xxx}+\frac{w_1}{\alpha^2}\varphi^2\varphi_x+\frac{w_2}{\alpha^3}\varphi^3\varphi_x=2\varphi_x\varphi_{xx}+\varphi\varphi_{xxx}.
\end{equation}
Among all the traveling wave solutions of  the R-CH equation \eqref{eq.10}, we shall focus on solutions which have the additional property that the waves are localized and that $\varphi$ and its derivatives decay at infinity, that is,
\begin{equation*}
\varphi^{(n)}(x) \rightarrow 0~~~~\text{as~$|x|$} \rightarrow \infty,~~~\text{for}~~ n\in \mathbb{N}.
\end{equation*}
Using the decay of $\varphi$ at infinity, integrating \eqref{eq.y} respect to spatial variable $x$ , one has\\
\begin{equation}\label{eq.20}
(c-\sigma)\varphi+\frac{3}{2}\varphi^2+\frac{w_1}{3\alpha^2}\varphi^3+\frac{w_2}{4\alpha^3}\varphi^4+\left(\sigma-\frac{\beta_0}{\beta}\right)\varphi_{xx}-\frac{1}{2}(\varphi_x)^2-\varphi\varphi_{xx}=0,
\end{equation}
then we can    rewrite \eqref{eq.20} as the following planar autonomous system
\begin{equation}\label{eq.2}
\begin{cases}
\varphi'=\zeta,\\
~~~\\
\zeta'=\dfrac{(c-\sigma)\varphi+\frac{3}{2}\varphi^2+\frac{w_1}{3\alpha^2}\varphi^3+\frac{w_2}{4\alpha^3}\varphi^4-\frac{1}{2}\zeta^2}{\varphi-\sigma+\frac{\beta_0}{\beta}}.
\end{cases}
\end{equation}
Our goal is to determine a homoclinic orbit in the phase plane starting and ending in (0,0) which corresponds to a solitary  wave solution of the R-CH equation \eqref{eq.10}. The existence of such an orbit depends on the wave speed.$\sigma$,  We start our analysis by determining the critical points of  system \eqref{eq.2}, that is, points where $(\varphi',\zeta')=(0,0)$. After
linearizing the system in the vicinity of those points to determine the local behavior, we
prove existence of a homoclinic orbit by analyzing the phase plane.  The main result for existence of solitary waves is as follows.
\begin{theorem}
If
\begin{align*}
\frac{\sqrt{2}}{4}<c<\sqrt{2}, c\neq1 ~and~ \sigma<c ,
\end{align*}
or
\begin{align*}
c>\sqrt{2} ~and~ \sigma>c,
\end{align*}
 then the R-CH equation \eqref{eq.10} exists solitary  wave solutions.
\end{theorem}

 \begin{proof}
 We talk about the two critical points of system \eqref{eq.2}: one at the origin $N_0=(0,0)$, and one given by $N_\sigma=(\varphi_\sigma,0)$, where $\varphi_\sigma$ is the unique real root of the
\begin{align}\label{eq.A}
T(\varphi)=\frac{w_2}{4\alpha^3}\varphi^3+\frac{w_1}{3\alpha^2}\varphi^2+\frac{3}{2}\varphi+c-\sigma.
\end{align}
Define that $\varphi=\eta-\frac{4}{9}\frac{\alpha w_1}{w_2}$, by the property of cubic polynomial, we can rewrite\\
\begin{equation}\label{eq.F}
T(\varphi)=T(\eta)=\frac{w_2}{4\alpha^3}(\eta^3+3p\eta+2q),
\end{equation}
where
\begin{equation*}
p=\frac{2\alpha^3}{w_2}-\frac{16w_1^2\alpha^2}{81w_2^2},~~~q=\frac{64w_1^3\alpha^3}{729w_2^3}-\frac{4w_1\alpha^4}{3w_2^2}+\frac{2\alpha^3(c-\sigma)}{w_2}.
\end{equation*}
The determinator of equation $T(\eta)=0$ is defined by
\begin{equation*}
\Delta=q^2+p^3
     =\left(\frac{64w_1^3\alpha^3}{729w_2^3}-\frac{4w_1\alpha^4}{3w_2^2}+\frac{2\alpha^3(c-\sigma)}{w_2}\right)^2+\left(\frac{2\alpha^3}{w_2}-\frac{16w_1^3\alpha^2}{81w_2^2}\right)^3.
\end{equation*}
Note that $\Delta>0$, then \eqref{eq.F}  has exactly one real root.   We can use the Cardano's formula for
third-order polynomials to find that $\eta_1$, which takes the form of
\begin{align*}
 \eta_1=\sqrt[3]{-\frac{q}{2}+\sqrt {\triangle}}+\sqrt[3]{-\frac{q}{2}-\sqrt {\triangle}}.
\end{align*}
Hence, we have
\begin{align*}
 \varphi_\sigma=\eta_1-\frac{4}{9}\frac{\alpha w_1}{w_2}.
\end{align*}
Both fixed points lie in the right half-plane where $\varphi> 0$ and we expect physically relevant solitary waves of elevation.
According to the relationship between the roots and the coefficients of the cubic equation for $\varphi_\sigma>0$, we can obtain
\begin{equation*}
-\frac{4\alpha^3(c-\sigma)}{w_2}>0,
\end{equation*}
which implies\\
\begin{align*}
\begin{cases}
c-\sigma>0,\\
w_2<0,
\end{cases}\text{or}~~~~
\begin{cases}
c-\sigma<0,\\
w_2>0.
\end{cases}
\end{align*}
 Hence, we have
 \begin{equation*}
 \frac{\sqrt{2}}{4}<c<\sqrt{2}, ~c\neq1 ~and~ \sigma<c~~~or ~~~c>\sqrt{2}~and~ \sigma>c.
\end{equation*}

To linearize the system near its critical points we compute the Jacobian Matrix $J$ of  system \eqref{eq.2} and evaluate it at $N_0$ and $N_\sigma$. All fixed points lie on the horizontal axis of the phase plane, so the Jacobian takes the form
$$
J=
\left({\begin{array}{cc}
0&1\\
~~\\
J_\sigma&0\\
\end{array}}
\right),
$$
where $J_\sigma=\partial_\sigma\zeta'$. Since the trace of $J$ is zero, all eigenvalues at the critical points are of the form
$$\lambda^{\pm}=\pm\sqrt{J_\sigma}$$
and the behavior of the system in the vicinity of the fixed points depends on the sign of $J_\sigma$.  Notice that
\begin{align*}
 \sigma-\frac{\beta_0}{\beta}>0,
 \end{align*}
 then at $N_0$, we have
\begin{equation*}
J_\sigma|_{(0,0)}=\frac{\sigma-c}{\sigma-\frac{\beta_0}{\beta}}
\begin{cases}
<0~~~~\text{if}~~\sigma<c,\\
=0~~~~\text{if}~~\sigma=c,\\
>0~~~~\text{if}~~\sigma>c,\\
\end{cases}
\end{equation*}
the eigenvalues of $J$ at the origin are $\lambda^{\pm}_0=\pm \sqrt{\frac{\sigma-c}{\sigma-\frac{\beta_0}{\beta}}}$. Therefore, we have\\
\noindent{\rm(a)}
\begin{minipage}[t]{0.9\linewidth}
When $\sigma>c$, we get two distinct
real eigenvalues of opposing sign and hence $N_0$ is a saddle point for the linearized system.
\end{minipage}\\
\noindent{\rm(b)}
\begin{minipage}[t]{0.9\linewidth}
When $\sigma<c$, the number $J_\sigma$ is negative, and the eigenvalues are purely imaginary. Hence, we can get $N_0$ is a center for the linearized system.
\end{minipage}\\

Evaluating $J_\sigma$ at the other critical point $N_\sigma=(\varphi_\sigma,0)$, we find that
\begin{equation*}
J_\sigma|_{(\varphi_\sigma,0)}=\frac{\varphi_\sigma T'(\varphi_\sigma)}{\varphi_\sigma-\sigma+\frac{\beta_0}{\beta}},
\end{equation*}
where $T(\varphi)$ is defined by \eqref{eq.A} and   has $\varphi_\sigma$ as its unique real root.
%Hence, $J_\sigma|_{(\varphi_\sigma,0)} < 0$ if and  only if $\varphi_\sigma>0$.
Important for our analysis is the fact that only when $\varphi_\sigma>0$, the fixed point $N_\sigma$ lies in the right half-plane where $\varphi > 0$.
In this case, we can hope to find a homoclinic orbit emerging and returning to the origin since $N_\sigma$ is a center whereas $N_0$ is a saddle point for the linearized system. Based on the above discussion, since $J_\sigma$ is nonzero whenever $\sigma\neq c$, both fixed points are hyperbolic which means that a (topological) saddle point for the linearized system remains a saddle also for the nonlinear system \cite{r35}.
%Since system \eqref{eq.2} is symmetric with respect to the horizontal axis, i.e. invariant under the transformation $(t,\zeta)\mapsto(-t,-\zeta)$, a linear center remains a center for the nonlinear system \cite{r35}. When $\sigma=c$ the two fixed points coincide at the origin and the Jacobian evaluated at $N_0$ reduces to

We look for a solution of  system \eqref{eq.2} which leaves the saddle point $N_0$ in the direction of the unstable Eigenspace spanned by the eigenvector $(1, \lambda^+_0 )$, encircles the center fixed point $N_\sigma$ and returns to the fixed point at the origin. This concludes the proof of existence of a homoclinic orbit starting and ending
in the origin which corresponds to a solitary  wave solution of \eqref{eq.20}, cf. Figure \ref{f1}. Hence,\\
\noindent{\rm(a)}
\begin{minipage}[t]{0.9\linewidth}
If $\frac{\sqrt{2}}{4}<c<\sqrt{2}$, $c\neq1$ and $\sigma<c$, the $N_\sigma=(\varphi_\sigma,0)$ lie in the right half-plane, there exists solitary  wave solutions of  the R-CH equation \eqref{eq.10}.
\end{minipage}\\
\noindent{\rm(a)}
\begin{minipage}[t]{0.9\linewidth}
If $c>\sqrt{2}$ and $\sigma>c$, the $N_\sigma=(\varphi_\sigma,0)$ lie in the right half-plane, there exists solitary  wave solutions of  the R-CH equation \eqref{eq.10}.
\end{minipage}\\
\end{proof}
\begin{figure}[H]
\centering
\includegraphics[height=6cm,width=7.5cm]{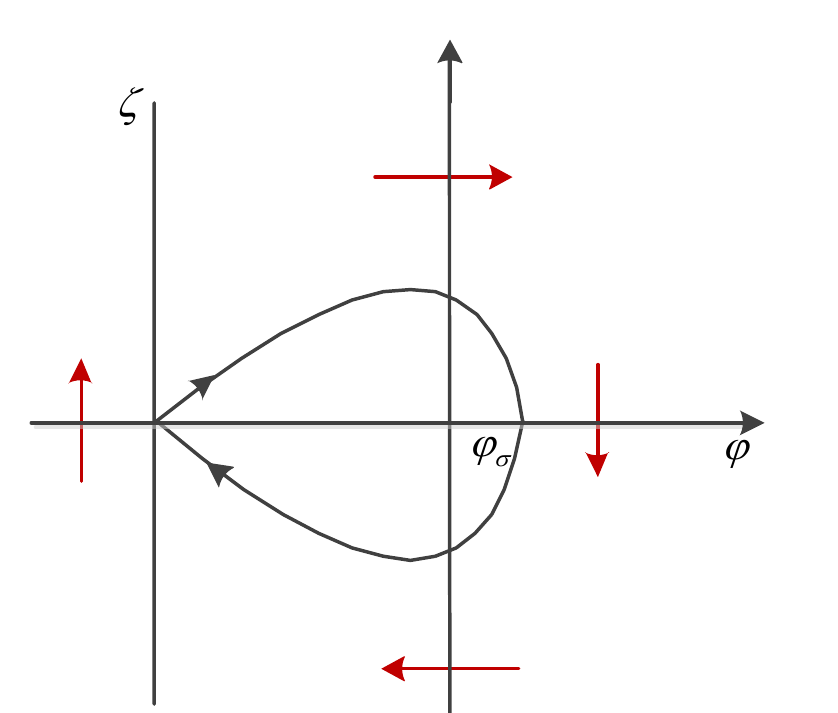}\\
\caption{Phase portrait of system (2.3) for $\sigma=2$ with a homoclinic orbit.}
\label{f1}
\end{figure}

\section{Stability of  Solitary Waves}\label{sec3}
In this section, we  prove that solitary  waves of the R-CH equation \eqref{eq.10} are orbitally stable.

If we define the orbit of a function $\varphi$ to be the set $O(\varphi) = \{\varphi(\cdot, x_0), x_0\in \mathbb{R}\}$, a solitary wave of the R-CH equation \eqref{eq.10} is called orbitally stable if profiles near its orbit remain at all later times near the orbit. More precisely:

\begin{definition}\label{D1}{\rm (Orbital stability)} The solitary waves $\varphi$ of the R-CH equation \eqref{eq.10} is stable if for every $\varepsilon>0$ there exists $\gamma>0$ such that if $u\in C([0,T)$; $H^2(\mathbb{R}))$ for some $0 < T\leq \infty$ is a solution to \eqref{eq.10} with $||u(0)-\varphi||_{H^2} \leq \gamma$, then for every $t \in [0,T)$, we have
\begin{align*}
\underset{\xi \in \mathbb{R}}{\inf}||u(t,.)-\varphi(.-\xi)||_{H^2}\leq \varepsilon.
\end{align*}
Otherwise, the solution is called unstable.
\end{definition}

\begin{theorem}\label{T1}\rm{(}\emph{See} \rm{\cite{r30}}\rm{)}\it
All solitray wave solutions $\varphi(x-\sigma t)$ of  Eq.\eqref{eq.10} are stable if and only if the scalar function
\begin{equation}\label{eq.55}
d(\sigma)=\sigma E(\varphi)-F(\varphi).
\end{equation}
is convex in a neighborhood of $\sigma$.
\end{theorem}

The following theorem is important to determine the sign of $d''(\sigma)$.

\begin{theorem}\label{T2}\rm{(}\emph{See} \rm{\cite{r322}}\rm{) } \it Let
\begin{equation*}
G_b(x)=g_n(b)x^n+g_{n-1}(b)x^{n-1}+\cdot\cdot\cdot+g_1(b)x+g_0(b)
\end{equation*}
be a family of real polynomials depending also polynomially on a real parameter $b$. If the following three conditions are met, then for all $b\in I$, $G_b(x)>0$. Suppose that there exists an open interval $I\subset\mathbb{R}$ such that:\\
~~\\
{\rm(i)} \it There is some $b_0\in I$, such that $G_{b_0}(x)>0$;\\
~\\
\rm{(ii)} \it For all $b\in I$, $\Delta_x(G_b)\neq 0$;\\
~\\
\rm{(iii)} \it For all $b\in I$, $g_n(b)\neq 0$.
\end{theorem}
As usual, we write $\Delta_x(G)$ to denote the discriminant of a polynomial $G(x)=a_nx^n+\cdot\cdot\cdot+a_1x+a_0$, that is, $$\Delta_x(G)=(-1)^\frac{n(n-1)}{2}\frac{1}{a_n}\text{Res}\left(G(x),G'(x)\right),$$
where $\text{Res}(G,G')$ is the resultant of $G$ and $G'$.

\begin{theorem}\label{th.w}
All solitary waves of the R-CH equation \eqref{eq.10} are stable.
\end{theorem}
\begin{proof}
We  show that the problem falls within the framework of the general approach to nonlinear stability provided by \cite{r30}. We compute the first and second variational derivatives of the conserved quantities $E$ and $F$, which are given by
\begin{equation*}
E'(u)v=(u-u_{xx},v),
\end{equation*}
\begin{equation*}
E''(u)v=(I-\partial_x^2,v),
\end{equation*}
\begin{equation*}
F'(u)v=\left(cu+\frac{3}{2}u^2-\frac{\beta_0}{\beta}u_{xx}+\frac{w_1}{3\alpha^2}u^3+\frac{w_2}{4\alpha^3}u^4-\frac{1}{2}u_x^2-uu_{xx},v\right),
\end{equation*}
and
\begin{equation*}
F''(u)v=\left(c+3u-\frac{\beta_0}{\beta}\partial_x^2+\frac{w_1}{\alpha^2}u^2+\frac{w_2}{\alpha^3}u^3-u_x\partial_x-u\partial_x^2-u_{xx},v\right).
\end{equation*}
In terms of the functionals $E$ and $F$, Eq.\eqref{eq.20}
can be written as
\begin{equation*}\label{eq.50}
\sigma E'(\varphi)-F'(\varphi)=0,
\end{equation*}
where $E'$ and $F'$ are the Fr\'{e}chet derivatives of $E$ and $F$ respectively in $H^1(\mathbb{R})$. The
linearized Hamiltonian operator $H_\sigma$ of $(\sigma E'-F')$ around $\varphi$ is defined by
\begin{align}
\notag H_\sigma&=\sigma E''(\varphi)-F''(\varphi)\\
 \notag  &=-\left(\sigma-\frac{\beta_0}{\beta}-\varphi\right)\partial_x^2-3\varphi-\frac{w_1}{\alpha^2}\varphi^2+\sigma+\frac{w_2}{\alpha^3}\varphi^3+\varphi_{xx}-c\\
         &=-\partial_x\left(\sigma-\frac{\beta_0}{\beta}-\varphi\right)\partial_x+\sigma-c-3\varphi-\frac{w_1}{\alpha^2}\varphi^2+\frac{w_2}{\alpha^3}\varphi^3+\varphi_{xx}.
\end{align}
Therefore, we can write the spectral equation $H_\sigma v=\lambda v$ as the Sturm-Liouville problem
\begin{equation*}
-(Pv_x)_x+(Q-\lambda)v=0,
\end{equation*}
where\\
$$P(x)=\sigma-\frac{\beta_0}{\beta}-\varphi,~~Q(x)=\sigma-c-3\varphi-\frac{w_1}{\alpha^2}\varphi^2+\frac{w_2}{\alpha^3}\varphi^3+\varphi_{xx}.$$
As is well-known, any regular Sturm-Liouville system has an infinite sequence of real eigenvalues $\lambda_0<\lambda_1<\lambda_2<\cdot\cdot\cdot$ with $\lim\limits_{n\to \infty}\lambda_n=\infty$ (see \cite{R31}). The eigenfunction $v_n(x)$ belonging to the eigenvalue $\lambda_n$ is uniquely determined up to a constant factor and has exactly $n$ zeros. In addition, we find that $H_\sigma$ is a self-adjoint with essential spectrum is given by $[\sigma-c,\infty)$ in view of the fact that $\lim\limits_{x\to \infty}{\rm\inf }~Q(x)=\sigma-c$ ( see \cite{R32}). It is very easy for us to verify that \eqref{eq.20} is equivalent to $H_\sigma(\varphi_x)=0$, and we know that
 %$\varphi_x$ has exactly one zero on $\mathbb{R}$ in view of the fact that the solitary wave solutions of \eqref{eq.20}
 the solitary wave solution of \eqref{eq.20} has a unique maximum, so the $\varphi_x$  has exactly one zero over $\mathbb{R}$.
We conclude that with only one negative eigenvalue, while the rest of the spectrum is positive and bounded away from zero. This proves what we said.

According to Theorem \ref{T1}, we know that the stability would be ensured by the convexity of the scalar function $d(\sigma)$ , and
differentiating $d(\sigma)$ with respect to $\sigma$ , we find that
\begin{equation}\label{eq.D}
d'(\sigma)=E(\varphi)+\left(\sigma E'(\varphi)-F'(\varphi),\partial_\sigma\varphi\right)=E(\varphi).
\end{equation}
In the following, we prove that $d''(\sigma)>0$ to infer that stability of $\varphi$. Let $\varphi$ be a solitary wave solution of \eqref{eq.10}, notice that
\begin{equation*}
\varphi_x^2=\varphi^2\frac{\frac{w_2}{10\alpha^3}\varphi^3+\frac{w_1}{6\alpha^2}\varphi^2+\varphi+c-\sigma}{\varphi-\sigma+\frac{\beta_0}{\beta}},
\end{equation*}
and that it is symmetric with respect to the crest. Therefore, $\varphi_x<0$ on $(0,\infty)$,   we have
\begin{equation*}
\varphi=-\varphi_x\sqrt{\frac{Y(\varphi,\sigma)}{S(\varphi,\sigma)}},
\end{equation*}
where
\begin{align*}
 S(\varphi,\sigma)=\frac{w_2}{10\alpha^3}\varphi^3+\frac{w_1}{6\alpha^2}\varphi^2+\varphi+c-\sigma
\end{align*}
and
\begin{align*}
Y(\varphi,\sigma)=\varphi-\sigma+\frac{\beta_0}{\beta}.
\end{align*}
Thus,  we have
\begin{align}\label{eq.60}
\notag d'(\sigma)&=\frac{1}{2} \int_{\mathbb{R}} (\varphi+\varphi_x^2)dx=\int_{0}^\infty \varphi^2\left(1+\frac{S}{Y}(\varphi,\sigma)\right)dx\\
\notag&=-\int_{0}^\infty \varphi\varphi_x\left(\frac{S+Y}{\sqrt{SY}}(\varphi,\sigma)\right)dx\\
&=\int_{0}^{h(\sigma)} y\left(\frac{S+Y}{\sqrt{SY}}(y,\sigma)\right)dy.
\end{align}
In order to determine the sign of $d''(\sigma)$, we adopt a method in \cite{RRD}.
 The transformation $y=\varphi(x)$ is introduced in \eqref{eq.60} and used the fact that $\varphi$ has a unique maximum $h(\sigma)$, which corresponds to the unique real root of $S(\varphi,\sigma)$. Denoting $H=h(\sigma)$, we find that
\begin{align*}
\sigma=h^{-1}(H)=\frac{w_2}{10\alpha^3}H^3+\frac{w_1}{6\alpha^2}H^2+H+c,
\end{align*}
which allows us to rewrite \eqref{eq.60} as
\begin{align}
\notag U(H)&=\int_{0}^{H} y\left(\frac{S+Y}{\sqrt{SY}}\left(y,h^{-1}(H)\right)\right)dy\\
\notag &=\int_{0}^{1} Hz\left(\frac{S+Y}{\sqrt{SY}}\left(Hz,h^{-1}(H)\right)\right)dz,
\end{align}
where we introduced the transformation $y=Hz$. This improper integral is well defined on $(0,1)$ since the polynomials in the square root are positive in this interval, and the integrand becomes singular only when $z=1$. We denote the integrand by $f(z,H)$ and let $I=[H_1,H_2]$ with $H_1>0$, so that we can view $U(H)=\int_{0}^{1} f(z,H)dz$ as a parameter integral.
Observe that $f(\cdot,H)\in L^1(0,1)$ for all $H\in I$, and $f(z,\cdot)\in C^1(I)$ for all $z\in (0,1)$. Moreover, we find that there exists a function
$g\in L^1(0,1) $ such that $|\partial_Hf(z,H)| \leq g(z)$ for all $(z,H)\in (0,1)\times I$. Indeed,
\begin{equation*}
f(z,H)=H^2z\left(\frac{S+Y}{\sqrt{SY}}\left(Hz,h^{-1}(H)\right)\right),
\end{equation*}
where
\begin{equation*}
S\left(Hz,h^{-1}(H)\right)=H(1-z)\left[\frac{w_2}{10\alpha^3}H^2(1+z+z^2)+\frac{w_1}{6\alpha^2}H(1+z)+1\right],
\end{equation*}
and
\begin{equation*}
Y\left(Hz,h^{-1}(H)\right)=-c-\frac{\beta_0}{\beta}+H(z-1)-\frac{w_1}{6\alpha^2}H^2-\frac{w_2}{10\alpha^3}H^3.
\end{equation*}
Differentiating $f(z,H)$ with respect to $H$, we can obtain a positive constant $k$ that depends only on $I$ such that
\begin{align*}
|\partial_HU(H)| \leq k\frac{1}{\sqrt{1-z}},
\end{align*}
for $z\in (0,1)$ and $H\in I$. Let
$g(z)=\frac{k}{\sqrt{1-z}}$ and observe that $g(z) \in  L^1(0,1)$ which proves
the claim. According to the theorem on differentiation of parameter integrals in \cite{r33}, we can obtain
\begin{align*}
\partial_HU(H)=\int_{0}^{1} \partial_Hf(z,H)dz.
\end{align*}
Since the choice of $I$ was arbitrary, we can generalize this result to all $H\in (0,\infty)$. Notice  that $H=h(\sigma)$, then
\begin{align*}
d''(\sigma)=\partial_HU(H)h'(\sigma),
\end{align*}
where $h'(\sigma)>0$ denotes the derivative of amplitude $h$ with respect to wave speed $\sigma$.
To this end,  let
\begin{align*}
A=\frac{w_2}{10\alpha^3}~~~~B=\frac{w_1}{6\alpha^2}~~~~K=\frac{\beta_0}{\beta},
\end{align*}
consider the integrand of this expression and notice that it may be rewritten as
\begin{align*}
\partial_HU(H)=\frac{H^2z(1-z)N(z,H)}{\sqrt{(SY)^3}},
\end{align*}
where
\begin{equation*}
\begin{aligned}
\begin{split}
N(z,H)=&\frac{1}{2}(4\, A^3\, H^8\, z^5 + 4\, A^3\, H^8\, z^4 + 4\, A^3\, H^8\, z^3 + 5\, A^2\, B\, H^7\, z^5 + 12\, A^2\, B\, H^7\, z^4 \\
&+ 12\, A^2\, B\, H^7\, z^3 + 2\, A^2\, B\, H^7\, z^2 - 6\, A^2\, H^6\, z^6 + 12\, A^2\, H^6\, z^3 - 7\, A^2\, H^5\, K\, z^5 \\
&- 7\, A^2\, H^5\, K\, z^4 - 7\, A^2\, H^5\, K\, z^3 + 2\, A^2\, H^5\, K\, z^2 + 2\, A^2\, H^5\, K\, z + 2\, A^2\, H^5\, K \\
&+ 7\, A^2\, H^5\, c\, z^5 + 7\, A^2\, H^5\, c\, z^4 + 7\, A^2\, H^5\, c\, z^3 - 2\, A^2\, H^5\, c\, z^2 - 2\, A^2\, H^5\, c\, z\\
&- 2\, A^2\, H^5\, c + 9\, A\, B^2\, H^6\, z^4 + 12\, A\, B^2\, H^6\, z^3 + 6\, A\, B^2\, H^6\, z^2 - 11\, A\, B\, H^5\, z^5 \\
&+ 14\, A\, B\, H^5\, z^3 + 8\, A\, B\, H^5\, z^2- 13\, A\, B\, H^4\, K\, z^4-13\, A\, B\, H^4\, K\, z^3 \\
&-4\, A\, B\, H^4\, K\, z^2 + 2\, A\, B\, H^4\, K\, z + 2\, A\, B\, H^4\, K + 13\, A\, B\, H^4\, c\, z^4 \\
&+13\, A\, B\, H^4\, c\, z^3 + 4\, A\, B\, H^4\, c\, z^2 - 2\, A\, B\, H^4\, c\, z - 2\, A\, B\, H^4\, c - 8\, A\, H^4\, z^4\\
& + 8\, A\, H^4\, z^3 - 9\, A\, H^3\, K\, z^3 + 9\, A\, H^3\, c\, z^3 + A\, H^2\, K^2\, z^2 + A\, H^2\, K^2\, z \\
&+ A\, H^2\, K^2- 2\, A\, H^2\, K\, c\, z^2 - 2\, A\, H^2\, K\, c\, z - 2\, A\, H^2\, K\, c + A\, H^2\, c^2\, z^2 \\
&+A\, H^2\, c^2\, z + A\, H^2\, c^2+ 4\, B^3\, H^5\, z^3 + 4\, B^3\, H^5\, z^2 - 5\, B^2\, H^4\, z^4 + 10\, B^2\, H^4\, z^2 \\
&- 6\, B^2\, H^3\, K\, z^3 - 6\, B^2\, H^3\, K\, z^2 + 6\, B^2\, H^3\, c\, z^3 + 6\, B^2\, H^3\, c\, z^2 - 6\, B\, H^3\, z^3\\
 &+ 6\, B\, H^3\, z^2 - 6\, B\, H^2\, K\, z^2 - 2\, B\, H^2\, K + 6\, B\, H^2\, c\, z^2 + 2\, B\, H^2\, c + 2\, B\, H\, K^2\, z \\
&+ 2\, B\, H\, K^2 - 4\, B\, H\, K\, c\, z - 4\, B\, H\, K\, c + 2\, B\, H\, c^2\, z + 2\, B\, H\, c^2 \\
&+ 2\, H\, K\, z - 2\, H\, K - 2\, H\, c\, z + 2\, H\, c + 3\, K^2 - 6\, K\, c + 3\, c^2).
\end{split}
\end{aligned}
\end{equation*}
Since $z\in (0,1)$ and $H\in (0,\infty)$, we just need to prove that $N(z,H)\geq 0$.  Theorem \ref{T2} ensures that the polynomials do not change sign. Let's first introduce the transformation $z=\frac{x^2}{1+x^2}$ and $H=b^2$, which maps the strip $(0,1)\times(0,\infty)$ into the whole plane. Since the denominator of the resulting expression is positive everywhere, we only consider the
numerator which is of the form
\begin{align}\label{eq.40}
G_b(x)=g_{12}(b)x^{12}+g_{10}(b)x^{10}+\cdot\cdot\cdot +g_{2}(b)x^{2}+g_{0}(b),
\end{align}
where the $g_i( i=0, 2, 4, 6, 8, 10, 12)$ (see Appendix) depend polynomially on the real parameter $b$. We can check
that $G_1(x)>0$ on $\mathbb{R}$ by the Sturm method, which ensures that the
assumption (i) in Theorem \ref{T2} is satisfied. Next, we compute the discriminant of $G_b(x)$ and find that it is nonzero, which ensures that the
assumption (ii) in Theorem \ref{T2} is satisfied. Finally, we use the Sturm method once more to show that $g_{12}(b)>0$
for all $b\in \mathbb{R}$, which ensures that the assumption (iii) in Theorem \ref{T2} is satisfied. Hence, we may conclude that $G_b(x)>0$ on $\mathbb{R}$.
In summary, this proves that $N(z,H)\geq0$ and we have $d''(\sigma)>0$, this completes the proof of theorem.

\end{proof}

\section*{Appendix}
The $ g_i (i=0, 2, 4, 6, 8, 10, 12)$  in \eqref{eq.40} are as follows
\begin{equation*}
\begin{split}
g_0(b)=&2\, c - 2\, K + 2\, B\, c - 6\, K\, c + 2\, A^2\, K + 2\, B\, K^2 + A\, c^2 - 2\, A^2\, c + 2\, B\, c^2 + 3\, K^2 + 3\, c^2 \\
&- 2\, A\, B\, c - 2\, A\, K\, c - 4\, B\, K\, c + A\, K^2\, c,
\end{split}
\end{equation*}
\begin{equation*}
\begin{split}
g_2(b)=&16\, A^2\, K\, b^{10} - 16\, A^2\, b^{10}\, c + 14\, A\, B\, K\, b^8 - 16\, A\, B\, b^8\, c + A\, K^2\, b^4\, c + 7\, A\, K^2\, b^4 \\
&- 16\, A\, K\, b^4\, c + 8\, A\, b^4\, c^2 + 14\, B^2\, b^2\, c^2 + 16\, B\, K^2\, b^2 - 12\, B\, K\, b^4 - 32\, B\, K\, b^2\, c\\
 &+ 14\, B\, b^4\, c+ 2\, B\, b^2\, c^2 + 21\, K^2 - 2\, K\, b^2 - 42\, K\, c + 12\, b^2\, c + 21\, c^2,
\end{split}
\end{equation*}
\begin{equation*}
\begin{split}
g_4(b)=&2\, A^2\, B\, b^{14} + 42\, A^2\, K\, b^{10} - 42\, A^2\, b^{10}\, c + 6\, A\, B^2\, b^{12} + 36\, A\, B\, K\, b^8 + 8\, A\, B\, b^{10}\\
 &- 36\, A\, B\, b^8\, c + 21\, A\, K^2\, b^4 - 42\, A\, K\, b^4\, c + 21\, A\, b^4\, c^2 + 4\, B^3\, b^{10} - 6\, B^2\, K\, b^6 \\
 &+ 10\, B^2\, b^8+ 6\, B^2\, b^6\, c + 40\, B\, K^2\, b^2 - 36\, B\, K\, b^4 - 80\, B\, K\, b^2\, c + 6\, B\, b^6 + 36\, B\, b^4\, c\\
 & + 40\, B\, b^2\, c^2 + 45\, K^2 - 20\, K\, b^2 - 90\, K\, c + 20\, b^2\, c + 45\, c^2,
\end{split}
\end{equation*}
\begin{equation*}
\begin{split}
g_6(b)=&4\, A^3\, b^{16} + 20\, A^2\, B\, b^{14} + 61\, A^2\, K\, b^{10} + 12\, A^2\, b^{12} - 61\, A^2\, b^{10}\, c + 36\, A\, B^2\, b^{12} \\
&+ 31\, A\, B\, K\, b^8 + 46\, A\, B\, b^{10} - 31\, A\, B\, b^8\, c + 34\, A\, K^2\, b^4 - 9\, A\, K\, b^6 - 68\, A\, K\, b^4\, c \\
&+ 8\, A\, b^8 + 9\, A\, b^6\, c + 34\, A\, b^4\, c^2 + 20\, B^3\, b^{10} - 30\, B^2\, K\, b^6 + 40\, B^2\, b^8 + 30\, B^2\, b^6\, c \\
&+ 60\, B\, K^2\, b^2 - 64\, B\, K\, b^4 - 120\, B\, K\, b^2\, c + 18\, B\, b^6 + 64\, B\, b^4\, c + 60\, B\, b^2\, c^2 \\
&+ 60\, K^2 - 20\, K\, b^2- 120\, K\, c + 20\, b^2\, c + 60\, c^2,
\end{split}
\end{equation*}
\begin{equation*}
\begin{split}
g_8(b)=&16\, A^3\, b^{16} + 60\, A^2\, B\, b^{14} + 34\, A^2\, K\, b^{10} + 36\, A^2\, b^{12} - 34\, A^2\, b^{10}\, c + 81\, A\, B^2\, b^{12} \\
&- 26\, A\, B\, K\, b^8 + 90\, A\, B\, b^{10} + 26\, A\, B\, b^8\, c + 31\, A\, K^2\, b^4 - 27\, A\, K\, b^6 - 62\, A\, K\, b^4\, c\\
& + 16\, A\, b^8+ 27\, A\, b^6\, c + 31\, A\, b^4\, c^2 + 36\, B^3\, b^{10} - 54\, B^2\, K\, b^6 + 55\, B^2\, b^8 + 54\, B^2\, b^6\, c \\
&+ 50\, B\, K^2\, b^2 - 66\, B\, K\, b^4 - 100\, B\, K\, b^2\, c + 18\, B\, b^6 + 66\, B\, b^4\, c + 50\, B\, b^2\, c^2 + 45\, K^2 \\
&- 10\, K\, b^2 90\, K\, c + 10\, b^2\, c + 45\, c^2,
\end{split}
\end{equation*}
\begin{equation*}
\begin{split}
g_{10}(b)=&24\, A^3\, b^{16} + 73\, A^2\, B\, b^{14} - 12\, A^2\, K\, b^{10} + 36\, A^2\, b^{12} + 12\, A^2\, b^{10}\, c + 78\, A\, B^2\, b^{12} \\
&- 59\, A\, B\, K\, b^8 + 63\, A\, B\, b^{10} + 59\, A\, B\, b^8\, c + 15\, A\, K^2\, b^4 - 27\, A\, K\, b^6 - 30\, A\, K\, b^4\, c \\
&+ 8\, A\, b^8+ 27\, A\, b^6\, c + 15\, A\, b^4\, c^2 + 28\, B^3\, b^{10} - 42\, B^2\, K\, b^6 + 30\, B^2\, b^8 + 42\, B^2\, b^6\, c \\
&+ 22\, B\, K^2\, b^2 - 36\, B\, K\, b^4 - 44\, B\, K\, b^2\, c + 6\, B\, b^6 + 36\, B\, b^4\, c + 22\, B\, b^2\, c^2 + 18\, K^2 \\
&- 2\, K\, b^2- 36\, K\, c + 2\, b^2\, c + 18\, c^2,
\end{split}
\end{equation*}
\begin{equation*}
\begin{split}
g_{12}(b)=&12\, A^3\, b^{16} + 31\, A^2\, B\, b^{14} - 15\, A^2\, K\, b^{10} + 6\, A^2\, b^{12} + 15\, A^2\, b^{10}\, c + 27\, A\, B^2\, b^{12} \\
&- 26\, A\, B\, K\, b^8 + 11\, A\, B\, b^{10} + 26\, A\, B\, b^8\, c + 3\, A\, K^2\, b^4 - 9\, A\, K\, b^6 - 6\, A\, K\, b^4\, c \\
& + 9\, A\, b^6\, c + 3\, A\, b^4\, c^2 + 8\, B^3\, b^{10} - 12\, B^2\, K\, b^6 + 5\, B^2\, b^8 + 12\, B^2\, b^6\, c + 4\, B\, K^2\, b^2\\
& - 8\, B\, K\, b^4 - 8\, B\, K\, b^2\, c + 8\, B\, b^4\, c + 4\, B\, b^2\, c^2 + 3\, K^2 - 6\, K\, c + 3\, c^2.
\end{split}
\end{equation*}

%\section*{Declarations}

%\section*{Authors' contributions}
%All authors wrote and reviewed the manuscript.

\section*{Acknowledgements}
This work is supported by Yunnan Fundamental Research Projects (Grant NO. 202101BE070001-050).

\section*{Conflict of interest statement }
This work does not have any conflicts of interest.

%\section*{Availability of data }
%No data was used for the research described in the article.


\begin{thebibliography}{99}

\bibitem{r32a}
Andronov A A, Leontovich E A, Gordon I I, Maier A G. Qualitative Theory of Second Order Dynamic Systems. John Wiley and Sons, New York, 1973.





\bibitem{r24}
Bressan A, Constantin A. Global conservative solutions of the Camassa-Holm equation. Archive for Rational Mechanics and Analysis, 2007, 183(2): 215-239.

\bibitem{r25}
Bressan A, Constantin A. Global dissipative solutions of the Camassa-Holm equation. Analysis and Applications, 2007, 5(1): 1-27.

\bibitem{R31}
Birkhoff G, Rota G C. Ordinary Differential Equations. John Wiley and Sons, New York, 1998.

\bibitem{r1}
Chen R M, Gui G, Liu Y. On a shallow-water approximation to the Green-Naghdi equations with the Coriolis effect. Advances in Mathematics, 2018, 340: 106-137.

\bibitem{r6}
Constantin A, Johnson R S. The dynamics of waves interacting with the equatorial undercurrent. Geophysical and Astrophysical Fluid Dynamics, 2015, 109(4): 311-358.

\bibitem{r7}
Constantin A, Johnson R S. An exact, steady, purely azimuthal equatorial flow with a free surface. Journal of Physical Oceanography, 2016, 46(6): 1935-1945.

\bibitem{r21}
Constantin A, Strauss W A. Stability of the Camassa-Holm solitons. Journal of Nonlinear Science, 2002, 12: 415-422.

\bibitem{r23}
Constantin A, Strauss W A. Stability of peakons. Communications on Pure and Applied Mathematics, 2000, 53(5): 603-610.


\bibitem{r29}
Constantin A, Escher J. Wave breaking for nonlinear nonlocal shallow water equations. Acta Mathematica, 1998, 181(2): 229-243.

\bibitem{R30}
Constantin A, Escher J. Global existence and blow-up for a shallow water equation. Annali della Scuola Normale Superiore di Pisa-Classe di Scienze, 1998, 26(2): 303-328.

\bibitem{r31}
Constantin A. Existence of permanent and breaking waves for a shallow water equation: a geometric approach. Annales De L Institut Fourier, 2000, 50(2): 321-362.


\bibitem{r32}
Constantin A, Escher J. On the blow-up rate and the blow-up set of breaking waves for a shallow water equation. Mathematische Zeitschrift, 2000, 233(1): 75-91.



\bibitem{R9}
Camassa R, Holm D D. An integrable shallow water equation with peaked solitons. Physical Review Letters, 1993, 71(11): 1661-1664.



\bibitem{r27}
Danchin R. A few remarks on the Camassa-Holm equation. Differential and Integral Equations, 2001, 14(8): 953-988.

\bibitem{r28}
Danchin R. A note on well-posedness for Camassa-Holm equation. Journal of Differential Equations, 2003, 192(2): 429-444.


\bibitem{R32}
Dunford N, Schwarz J T. Linear Operators. John Wiley and Sons, New York, London, 1963.

\bibitem{r34}
Da Silva P L, Freire I L. Well-posedness travelling waves and geometrical aspects of generalizations of the Camassa-Holm equation. Journal of Differential Equations, 2019, 267(9): 5318-5369.

\bibitem{RRD}
Duruk Mutluba N, Geyer A. Orbital stability of solitary waves of moderate amplitude in shallow water. Journal of Differential Equations, 2013, 255(2): 254-263.


\bibitem{r33}
Escher J, Amann H. Analysis. Birkh\"auser, Basel, Boston, Berlin, 2001.

\bibitem{R10}
Fuchssteiner B, Fokas A S. Symplectic structures their B\"{a}cklund transformations and hereditary symmetries. Physica D: Nonlinear Phenomena, 1981, 4(1): 47-66.




\bibitem{r20}
Gui G, Liu Y, Sun J. A nonlocal shallow water model arising from the full water waves with the Coriolis effect. Journal of Mathematical Fluid Mechanics, 2019, 21(2): 1-30.

\bibitem{r2}
Gui G, Liu Y, Luo T. Model equations and traveling wave solutions for shallow water waves with the Coriolis effect. Journal of Nonlinear Science, 2019, 29(3): 993-1039.


\bibitem{r30}
Grillakis M, Shatah J, Strauss W A. Stability theory of solitary waves in the presence of symmetry. Journal of Functional Analysis, 1987, 74: 160-197.

\bibitem{r322}
Garcia-Saldana J D, Gasull A, Giacomini H. Bifurcation values for a family of planar vector fields of degree five. Discrete and Continuous Dynamical Systems, 2012, 35(2): 669-701.

\bibitem{R24}
Geyer A. Solitary traveling water waves of moderate amplitude. Journal of Nonlinear Mathematical Physics, 2012, 19: 104-115.
\bibitem{r26}
 Li Y A, Olver P J. Well-posedness and blow-up solutions for an integrable nonlinearly dispersive model wave equation. Journal of Differential Equations, 2000, 162(1): 27-63.


\bibitem{r35}
Perko L. Differential Equations and Dynamical Systems. Springer, New York, 1991.

\bibitem{r3}
Tu X, Liu Y, Mu C. Existence and uniqueness of the global conservative weak solutions to the rotation-Camassa-Holm equation. Journal of Differential Equations, 2019, 266(8): 4864-4900.


\bibitem{R35}
Wang T, Yang S, Han X. Symmetric waves are traveling waves for the rotation-Camassa-Holm equation. Journal of Mathematical Fluid Mechanics, 2021, 23(3): 1-4.

\bibitem{r17}
Yang S. Generic regularity of conservative solutions to the rotational-Camassa-Holm equation. Journal of Mathematical Fluid Mechanics, 2020, 22(4): 1-11.

\bibitem{r36}
Zhang L. Non-uniform dependence and well-posedness for the rotation-Camassa-Holm equation on the torus. Journal of Differential Equations, 2019, 267(9): 5049-5083.



\end{thebibliography}
\end{document}